\documentclass[11pt]{amsart}

\usepackage[colorlinks=true, pdfstartview=FitV, linkcolor=blue, 
citecolor=blue]{hyperref}

\usepackage{a4wide}
\usepackage{amsmath,amssymb} 
\usepackage{graphicx}

\theoremstyle{plain}
\newtheorem{theorem}{Theorem}[section]
\newtheorem{prop}[theorem]{Proposition}

\newtheorem{coro}[theorem]{Corollary}

\theoremstyle{definition}
\newtheorem{example}[theorem]{Example}
\newtheorem{remark}[theorem]{Remark}

\newcommand{\dd}{\,\mathrm{d}}

\newcommand{\pts}{\hspace{0.5pt}}
\newcommand{\nts}{\hspace{-0.5pt}}

\newcommand{\Leb}{\lambda^{}_{\mathrm{L}}}

\newcommand{\ZZ}{\mathbb{Z}\pts}
\newcommand{\QQ}{\mathbb{Q}}
\newcommand{\RR}{\mathbb{R}}
\newcommand{\NN}{\mathbb{N}}
\newcommand{\CC}{\mathbb{C}}
\newcommand{\TT}{\mathbb{T}}

\newcommand{\XX}{\mathbb{X}}
\newcommand{\YY}{\mathbb{Y}}

\newcommand{\cA}{\mathcal{A}}

\newcommand{\cO}{\mathcal{O}}

\newcommand{\cL}{\mathcal{L}}

\newcommand{\vL}{\varLambda}

\newcommand{\exend}{\hfill $\Diamond$}
\newcommand{\defeq}{\mathrel{\mathop:}=}

\DeclareMathOperator{\sinc}{sinc}
\DeclareMathOperator{\dens}{dens}

\DeclareMathOperator{\card}{card}

\DeclareMathOperator{\Mat}{Mat}
\DeclareMathOperator{\dotcup}{\dot\cup}

\newcommand{\oplam}{\mbox{\Large $\curlywedge$}}

\newcommand{\myfrac}[2]{\frac{\raisebox{-2pt}{$#1$}}
      {\raisebox{0.5pt}{$#2$}}}

\begin{document}

\title[Scaling of diffraction intensities 
near the origin]{Scaling of diffraction intensities 
near the origin:\\[2mm] Some rigorous results}

% \date{\today}

\author{Michael Baake}
\address{Fakult\"at f\"ur Mathematik, Universit\"at Bielefeld, \newline
\hspace*{\parindent}Postfach 100131, 33501 Bielefeld, Germany}
\email{mbaake@math.uni-bielefeld.de }

\author{Uwe Grimm}
\address{School of Mathematics and Statistics,
  The Open University,\newline \hspace*{\parindent}Walton Hall, 
  Milton Keynes MK7 6AA, United Kingdom} 
\email{uwe.grimm@open.ac.uk}

\begin{abstract}  
  The scaling behaviour of the diffraction intensity near the origin
  is investigated for (partially) ordered systems, with an emphasis on
  illustrative, rigorous results. This is an established method to
  detect and quantify the fluctuation behaviour known under the term
  hyperuniformity. Here, we consider one-dimensional systems with pure
  point, singular continuous and absolutely continuous diffraction
  spectra, which include perfectly ordered cut and project and
  inflation point sets as well as systems with stochastic disorder.
\end{abstract}

\maketitle

\section{Introduction}

While the concept of \emph{order} is intuitive, it is surprisingly
challenging to `measure' the order in a system in a way that leads to
a meaningful classification of different manifestations of order in
Nature. Some of the most prominent measures available are based on
ideas from statistical physics and crystallography, such as
\emph{entropy} which is closely related to disorder, or
\emph{diffraction} which is the main tool for determining the
structure of solids. Essentially, diffraction is the Fourier transform
of the pair correlation function, and hence quantifies the order in
the two-point correlations of the structure.

Following the discovery of quasicrystals in the early 1980s, a proper
mathematical treatment of the diffraction of aperiodically ordered
structures was required; see \cite{TAO,Hof} and references therein for
background and details. It turns out that there is a close connection
between diffraction and what is known as the \emph{dynamical spectrum}
in mathematics \cite{BL}. The latter is the spectrum of the
corresponding Koopman operator and is an important concept in ergodic
theory.  Aperiodically ordered structures, in particular those
constructed by inflation rules, provide interesting examples of
systems that exhibit a scaling (self-similarity) type of order, which
differs substantially from the translational order found in periodic
structures, such as conventional crystals.

Starting from the idea to use the degree of `(hyper)uniformity' in
density fluctuations in many-particle systems \cite{TS} to
characterise their order, the scaling behaviour of the diffraction
near the origin has emerged as a measure that captures the variance of
the long-distance correlations. Recently, a number of conjectures on
the scaling behaviour of the diffraction of aperiodically ordered
structures were made \cite{Josh1,Josh2}, reformulating and extending
earlier, partly heuristic, results by Luck \cite{Luck} from this
perspective; see also \cite{Aubry,GL}.

The purpose of this article is to link the recent interest in these
questions with some of the known techniques and results from rigorous
diffraction theory, as started by Hof in \cite{Hof} and later
developed by many people; see \cite{TAO} and references therein for a
systematic account. Our approach will make substantial use of the
exact renormalisation relations for primitive inflation rules
\cite{BFGR,BG15,BGM,NM,BaGriMa}, which will allow us to establish the
scaling behaviour rigorously. This is in some contrast to classic
methods of finite size scaling \cite{VR88,VR03,VR05,VR10}, where such
a behaviour is extrapolated and only asymptotically true.\medskip

The paper is organised as follows. We begin by recalling some
background material on diffraction, the projection formalism,
inflation rules and Lyapunov exponents in
Section~\ref{sec:prelim}. Then, in Section~\ref{sec:pp}, we discuss
systems with pure point spectrum, starting from the paradigmatic
Fibonacci chain, which we treat in two different ways. We also discuss
various generalisations, including noble means inflations and a
limit-periodic system. After a brief look into substitutions with more
than two letters, we close with an instructive example of
number-theoretic origin, which can also be understood as a projection
set.

In Section~\ref{sec:ac}, we summarise the situation for systems with
absolutely continuous spectrum via a comparative exposition of random
and deterministic cases. In particular, we discuss the Poisson process
in comparison to the Rudin--Shapiro sequence, the classical random
matrix ensembles, as well as the Markov lattice gas and binary random
tilings. Finally, Section~\ref{sec:sc} deals with systems with
singular continuous spectrum, of which the classic
\mbox{Thue{\pts}}--Morse sequence is a paradigm with a decay behaviour
faster than any power law. We then embed this into the family of
generalised \mbox{Thue{\pts}}--Morse sequences.

\section{Preliminaries and general methods}\label{sec:prelim}

Our approach to the scaling behaviour of the diffraction measure near
the origin requires a number of different methods. In this section, we
summarise key results and provide references for background and
further details. 

\subsection{Diffraction and scaling}

Throughout this article, we use the notation and results from
\cite{TAO} on diffraction theory, as well as some more advanced
results on the Fourier transform of measures from \cite{MS,TAO2}.
Various definitions are discussed there, and we use standard results
from these sources without further reference.  Note that the term
`measure' in this context refers to general (complex) Radon measures
in the mathematical sense. They can be identified with the continuous
linear functionals on the space of compactly supported continuous
functions on $\RR^d$.  In particular, given a (usually translation
bounded) measure $\omega$, we assume that an averaging sequence $\cA$
of van Hove type is specified, and the autocorrelation $\gamma =
\gamma^{}_{\omega} = \omega \circledast \widetilde{\omega}$ is given
as the Eberlein (or volume-averaged) convolution along $\cA$.  The
measure $\gamma$ is positive definite, and hence Fourier transformable
as a measure.

The Fourier transform of $\gamma$ is the \emph{diffraction measure}
$\widehat{\gamma}$, which is a positive measure. This is the
measure-theoretic formulation of the \emph{structure factor} from
physics and crystallography, which is better suited for rigorous
results. In particular, this approach defines the different spectral
components by means of the Lebesgue decomposition
\[
  \widehat{\gamma} \, = \,
  \widehat{\gamma}^{}_{\mathsf{pp}} +
  \widehat{\gamma}^{}_{\mathsf{sc}} +
  \widehat{\gamma}^{}_{\mathsf{ac}}
\]
of $\widehat{\gamma}$ into its pure point, singular continuous and
absolutely continuous parts; see \cite[Sec.~8.5.2]{TAO} for details.
In what follows, we only consider the one-dimensional case.

For the investigation of scaling properties, we follow the
existing literature and define
\begin{equation}\label{eq:Z-def}
     Z (k) \, \defeq \, \widehat{\gamma} \bigl( (0,k]\bigr) ,
\end{equation}
which is a modified version of the distribution function of the
diffraction measure. Due to the reflection symmetry of
$\widehat{\gamma}$ with respect to the origin, this quantity can also
be expressed as
\[
    Z (k) \, = \, \myfrac{1}{2}
    \Bigl(\widehat{\gamma} \bigl( [-k,k]\bigr) -
    \widehat{\gamma} \bigl( \{0\}\bigr) \Bigr).
\]
Sometimes, it is natural to replace $Z(k)$ by $Z(k) / I(0)$, where
$I(0) = \widehat{\gamma} \bigl( \{ 0 \} \bigr)$ is the intensity of
the central diffraction (or Bragg) peak.  Note that this just amounts
to a different normalisation. It will always be clear from the context
when we do so. This normalisation has no influence on the scaling
behaviour of $Z (k)$ as
$k\,\raisebox{2pt}{$\scriptscriptstyle \searrow$}\, 0$.

The interest in the scaling near the origin is based on the intuition
that the small-$k$ behaviour of the diffraction measure probes the
long-wavelength fluctuations in the structure, which is related to the
variance in the distribution of patches, and hence can serve as an
indicator for the degree of uniformity of the structure \cite{TS}.
Clearly, any periodic structure leads to $Z(k) = 0$ for all
sufficiently small $k$, so that the main interest is focused on
non-periodic systems, both ordered and disordered.

\subsection{Projection formalism}

One way to produce well-ordered aperiodic systems is based on the
(partial) projection of higher-dimensional (periodic) lattices.  Such
\emph{cut and project sets} or \emph{model sets} can be viewed as a
generalised variant of the notion of a quasiperiodic function. We can
only present a brief summary here, for details we refer to
\cite[Ch.~7]{TAO}.  The general setting for a model set in
\emph{physical} (direct) space $\RR^{d}$ is encoded in the \emph{cut
  and project scheme} (CPS) $(\RR^{d},H,\cL)$,
\begin{equation}\label{eq:cps}
\renewcommand{\arraystretch}{1.2}\begin{array}{r@{}ccccc@{}l}
   & \RR^{d} & \xleftarrow{\,\;\;\pi\;\;\,} 
         & \RR^{d} \, \times \, H  & 
        \xrightarrow{\;\pi^{}_{\mathrm{int}\;}} & H & \\
   & \cup & & \cup & & \cup & \hspace*{-2ex} 
   \raisebox{1pt}{\text{\footnotesize dense}} \\
   & \pi(\cL) & \xleftarrow{\; 1-1 \;} & \cL & 
   \xrightarrow{\; \hphantom{1-1} \;} & 
       \pi^{}_{\mathrm{int}}(\cL) & \\
   & \| & & & & \| & \\
   &  L & \multicolumn{3}{c}{\xrightarrow{\qquad\quad\quad \;\;
       \;\star\; \;\; \quad\quad\qquad}} 
       &   {L_{}}^{\star} & \\
\end{array}\renewcommand{\arraystretch}{1}
\end{equation}
where the \emph{internal} space $H$ is a locally compact Abelian group
(in many examples, $H$ will turn out to be another Euclidean space, so
$H=\RR^{m}$), $\cL$ is a lattice (co-compact discrete subgroup) in
$\RR^{d}\times H$, and where $\pi$ and $\pi^{}_{\mathrm{int}}$ denote
the natural projections onto the physical and internal spaces.  The
assumption that $L=\pi(\cL)\subset\RR^{d}$ is a bijective image of
$\cL$ in physical space guarantees that the $\star\,$-map $x\mapsto
x^{\star}$ is well defined on $L$.

A \emph{model set} for a given CPS is then a set of the form
\begin{equation}\label{eq:ms}
    \vL \, = \, \oplam(W) \, = \, 
    \bigl\{  x\in L :  x^{\star} \in W \bigr\} ,
\end{equation}
where the domain $W\subseteq H$ (called the \emph{window} or
\emph{acceptance domain}) is a relatively compact set with non-empty
interior. These conditions on the window guarantee that the model set
$\vL\subseteq L$ is both uniformly discrete and relatively
dense, so a Delone set in $\RR^{d}$. If the window is
sufficiently `nice' (for instance, when $W$ is compact with boundary
of measure $0$), the diffraction measure $\widehat{\gamma}$ of the
associated Dirac comb
\[
   \delta^{}_{\!\vL}\, \defeq  \sum_{x\in\vL} \delta^{}_{x}
\]
is a pure point measure; see \cite[Ch.~9]{TAO} for details. Note that
the dynamical spectrum is pure point as well; see \cite{BL} and
references therein.

\begin{example}\label{ex:Fibo}
  Consider a CPS with $d=1$ and $H=\RR$, and the planar lattice
\[
   \cL \, = \, \bigg\langle\!
   \binom{1}{1}, \binom{\tau}{1-\tau} 
   \!\bigg\rangle_{\!\!\ZZ} \subset\, \RR^{2}, 
\]
where $\tau=\bigl(1+\sqrt{5}\, \bigr)/2$ is the golden ratio. The
projection of $\cL$ to physical space is the $\ZZ$-module
$L=\ZZ[\tau]= \{m+n\tau\mid m,n\in\ZZ\}$, and similarly
$L^{\star} = \ZZ[\tau]$. The $\star$-map acts as algebraic
conjugation $\sqrt{5}\mapsto -\sqrt{5}$, so that
$(m+n\tau)^{\star}=m+n-n\tau$.  The \emph{Fibonacci model set} is
$\vL=\oplam(W)$ with window $W = (-1,\tau-1]$. As indicated in
Figure~\ref{fig:goldproj}, the two different point types (left
endpoints of long and short intervals, respectively) in the Fibonacci
model set are obtained as model sets for the two smaller windows
$W^{}_{\! L} = (-1,\tau-2]$ and $W^{}_{\! S} = (\tau-2,\tau-1]$, with
$W=W^{}_{\! L}\dotcup W^{}_{\! S}$.\exend
\end{example}

\begin{figure}
\centerline{\includegraphics[width=0.8\textwidth]{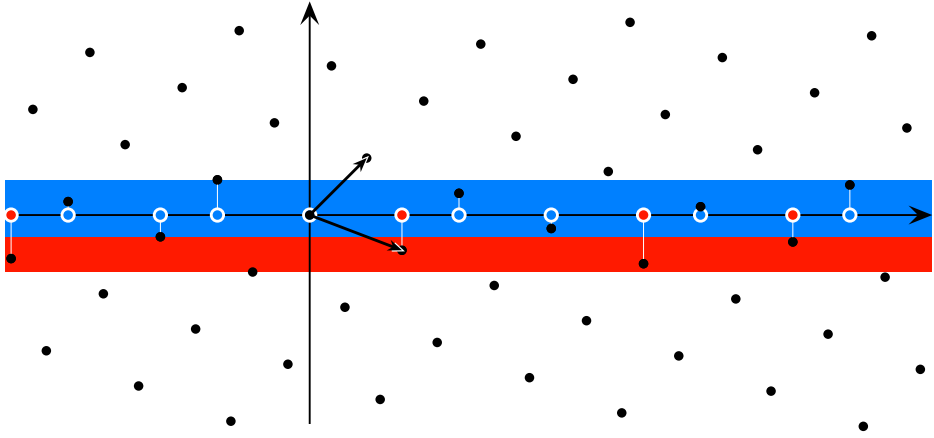}}
\caption{\label{fig:goldproj}Projection description of the Fibonacci
  model set from Example~\ref{ex:Fibo}. The horizontal line represents
  the physical space, the vertical line the internal space.  The black
  points form the lattice $\cL$. Points inside the strip
  $\RR\times W$ are projected to physical space to produce the
  model set $\vL$.}
\end{figure}

\subsection{Inflation rules and exact 
renormalisation}\label{subsec:renorm}

Another important construction of aperiodic sequences and tilings is
based on substitution and inflation rules; see \cite[Chs.~4--5]{TAO}
for background. The idea is perhaps most easily phrased in terms of
tilings. Starting from a (finite) set of prototiles (equivalence
classes of tiles under translation), an \emph{inflation rule} consists
of a rescaling (more generally, an expansive linear transformation) of
the prototiles and a subsequent dissection of the rescaled tiles into
prototiles of the original shape and size. Iterating such an inflation
rule produces tilings of space, which generically will not be
periodic. Let us illustrate this with the example of the Fibonacci
model set from above, which can also be constructed by an inflation
rule (note that this is not generally the case: Cut and project
tilings generically do \emph{not} possess an inflation symmetry, and
inflation tilings generally do not allow an embedding into a
higher-dimensional lattice with a bounded window).

\begin{example}\label{ex:fiboinf}
  Consider two prototiles (intervals) in $\RR$, a long interval $a$ of
  length $\tau$ and a short interval $b$ of length $1$. The
  (geometric) Fibonacci inflation rule with inflation factor $\tau$ is
\[
   \includegraphics[width=0.6\textwidth]{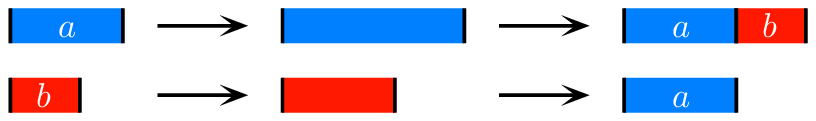}
\] 
Starting from an initial configuration of, say, two adjacent large
tiles with common vertex at the origin, and iterating this rule
produces a tiling of $\RR$ which is a fixed point under the
square of the inflation rule. Denoting the set of left endpoints of
intervals of type $a$ by $\vL_{a}$ and of type $b$ by
$\vL_{b}$, we have $\vL_{a,b}\subset \ZZ[\tau]$
by construction, which are model sets in the description of
Example~\ref{ex:Fibo}; see \cite[Ex.~7.3]{TAO} for details.
\exend
\end{example}

The inherent recursive structure of the inflation approach not only
leads to recursive relations for the point sets constructed in this
way, but also ensures the existence of a set of renormalisation
equations for their two-point correlation functions; see \cite{BG15}
and references therein for details. Here, we illustrate this for the
Fibonacci case.

\begin{example}\label{ex:fiborec}
  Consider the point set
  $\vL=\vL_{a}\dotcup\vL_{b}$ constructed by the
  inflation of Example~\ref{ex:fiboinf}.  Define the \emph{pair
    correlation functions} (or coefficients)
  $\nu^{}_{i j} (z)$ with $i,j \in \{ a,b\}$ as the
  \emph{relative} frequency (frequency per point of $\vL$) of
  two points in $\vL$ at distance $z$, subject to the condition that
  the left point is of type $i$ and the right point of type
  $j$.  The pair correlation functions exist (as a consequence of
  unique ergodicity) and satisfy the relations
  $\nu^{}_{i j} (z)=0$ for any
  $z\not\in \vL-\vL$, $\nu^{}_{i j} (0)=0$ for
  $i\neq j$, and
  $\nu^{}_{i j} (-z) = \nu^{}_{j i} (z)$ for any
  $z$.

The recognisability of our inflation rule then implies the following
set of exact renormalisation equations \cite{BG15,NM,BGM},
\begin{equation}\label{eq:Fibo-rec}
\begin{split}
   \nu^{}_{aa} (z) \, & = \, \tfrac{1}{\tau} \bigl(
     \nu^{}_{aa} (\tfrac{z}{\tau}) + \nu^{}_{ab} (\tfrac{z}{\tau})
    +\nu^{}_{ba} (\tfrac{z}{\tau}) + \nu^{}_{bb} (\tfrac{z}{\tau}) 
       \bigr) , \\
   \nu^{}_{ab} (z) \, & = \, \tfrac{1}{\tau} \bigl(
      \nu^{}_{aa} (\tfrac{z}{\tau} - 1) + 
      \nu^{}_{ba} (\tfrac{z}{\tau} - 1)  \bigr), \\
   \nu^{}_{ba} (z) \, & = \, \tfrac{1}{\tau} \bigl(
      \nu^{}_{aa} (\tfrac{z}{\tau} + 1) + 
      \nu^{}_{ab} (\tfrac{z}{\tau} + 1) \bigr), \\
   \nu^{}_{bb} (z) \, & = \, \tfrac{1}{\tau} \bigl( 
      \nu^{}_{aa} (\tfrac{z}{\tau}) \bigr),
\end{split}
\end{equation}
where $z\in \ZZ[\tau]$ and $\nu^{}_{ij} (z) = 0$ whenever
$z\not\in \vL_{j} - \vL_{i}$. This set of
equations has a unique solution if one of the single-letter
frequencies $\nu^{}_{aa} (0), \nu^{}_{bb} (0)$ is given; see
\cite{BG15} for details and further examples.

The autocorrelation measure for the point set $\vL$ is
of the form
\[
    \gamma^{}_{\vL} \, =  \sum_{z\in\vL - \vL} 
    \eta (z)\,  \delta_{z}
\]
with autocorrelation coefficients
\[
    \eta(z) \, \defeq \, \lim_{R\to\infty} 
    \frac{\card \bigl( \vL^{}_{R} \cap (z+\vL^{}_{R}) 
    \bigr)}{2R}  ,
\]
where $\vL^{}_{R} \defeq \vL\cap [-R,R]$. Again, these
coefficients exist by unique ergodicity. The autocorrelation
coefficients can be expressed in terms of the pair correlation
functions as
\begin{equation}\label{eq:rel-nu} 
   \frac{ \eta (z)}{\dens (\vL)} \, = \, 
    \nu^{}_{aa} (z) +  \nu^{}_{ab} (z) +  
    \nu^{}_{ba} (z) +  \nu^{}_{bb} (z) .
\end{equation}
By taking Fourier transforms, this allows us to use the
renormalisation approach to analyse the diffraction of
$\vL$; see \cite{BFGR,BaGriMa,BGM} for details. 
\exend
\end{example} 

The approach of Example~\ref{ex:fiborec} can be applied to inflation
tilings in quite a general setting, as it only requires
recognisability (not necessarily local) of the inflation
rule. Studying the corresponding diffraction spectrum and, in
particular, its spectral components, naturally leads to the
investigation of matrix iterations and their asymptotic behaviour; see
\cite{BFGR,BaGriMa,BGM} for recent work in this direction, as well as
\cite{BuSol,BuSol2} for related work in the more general setting of
matrix Riesz products for substitution systems with arbitrary choices
of length scales.

\subsection{Lyapunov exponents for (effective) single
    matrix iterations}\label{subsec:iter}

If $M \in \Mat (d, \CC)$ is a fixed matrix, the asymptotic behaviour
of $ \| M^n v \|$ as $n\to\infty$ with $v\in \CC^d$ can be understood
in terms of the Lyapunov exponents of this iteration, via determining
\[
    \lim_{n\to\infty} \myfrac{1}{n} \log \| M^{n} v\| \pts ;
\]
see \cite{BP,Viana} for background. In this simple single matrix case,
the \emph{Lyapunov spectrum} of $M$ is
\[
    \sigma^{}_{\mathrm{L}} (M) \, = \,
    \{ \log \pts \lvert \mu \rvert :
    \mu \in \sigma (M) \} \pts ,
\]
where $\sigma (M)$ denotes the spectrum of $M$; compare \cite{BP}. By
an expansion in the principal vectors and eigenvectors of $M$, one can
also extract the explicit scaling behaviour as $n\to\infty$, for any
given $v$.

Slightly more interesting and relevant to us is the following
situation. Consider a matrix family $B (k)$ for $k \geqslant 0$ that
is smooth (in fact, analytic) in $k$, with
$\lim_{k \searrow 0} B (k) = M$. Then, we are interested in the
asymptotic behaviour of
\[
   \| B(\lambda^{-n} k) B(\lambda^{-n+1} k) 
   \cdots B(\lambda^{-1} k)\pts v \|
\]
for fixed $v$ as $n\to\infty$, which means that we consider a specific
\emph{matrix cocycle} of $k \mapsto k/\lambda$ on $\RR$, with
$\lambda > 1$; see \cite{BP,Viana} for background. Due to the relation
with $M=B(0)$, the Lyapunov spectrum of the cocycle agrees with that of
$M$, so
\begin{equation}\label{eq:lspec}
  \lim_{n\to\infty} \myfrac{1}{n} \,
  \log  \| B(\lambda^{-n} k) B(\lambda^{-n+1} k) 
  \cdots B(\lambda^{-1} k) \pts v \| \, \in \,
  \sigma^{}_{\mathrm{L}} (M) \pts .
\end{equation}
This behaviour emerges from the observation that the cocycle, with
increasing length but fixed $k$, more and more looks like a
multiplication by powers of $M$.

\section{Aperiodic systems with pure point spectrum}\label{sec:pp}

Let us set the scene by considering a paradigmatic example in some
detail, which is actually approachable by two different methods,
namely by the projection formalism and by the renormalisation approach.
Common to all examples in this section is the fact that $Z(k)$
is not a continuous function of $k$, which requires some extra care.

\subsection{The Fibonacci chain and related systems}

Consider the primitive substitution $\varrho^{}_{\mathrm{F}}$ on the
binary alphabet $\cA = \{ a,b\}$, as defined by $a \mapsto ab$ and $b
\mapsto a$, which is the symbolic substitution that underlies
Example~\ref{ex:Fibo}. The Abelianisation leads to the corresponding
\emph{substitution matrix}
\[
  M^{}_{\mathrm{F}} \, = \, \begin{pmatrix}1 & 1 \\ 1 & 0
  \end{pmatrix}.
\]
Its Perron--Frobenius (PF) eigenvalue is the golden mean, $\tau$. The
corresponding right and left eigenvectors code the relative
frequencies of letters and the natural interval lengths for the
geometric representation, respectively.

There are two bi-infinite fixed points for $\varrho^{2}_{\mathrm{F}}$,
where the legal seeds are $a|a$ and $b|a$, with the vertical line
denoting the location of the origin; see \cite[Sec.~4]{TAO} for
details and background. We choose the one with seed $a|a$ and let
$\XX$ define the \emph{symbolic} hull generated by it via an orbit
closure under the shift action, which is actually independent of the
choice we made.  Now, we turn this $\XX$ into a \emph{tiling system}
by fixing natural tile (or interval) lengths according to the left PF
eigenvector, namely $\tau$ for $a$ and $1$ for $b$. The left endpoints
of the tiles, for the tiling that emerges from our chosen fixed point
with seed $a|a$, form a Delone set with distances $\tau$ and $1$, and
the orbit closure (in the local topology) under the (continuous)
translation action by $\RR$ gives the \emph{geometric} hull $\YY$ we
are interested in.
 
Let $\vL \in \YY$ be arbitrary, and consider the Dirac comb
$\delta^{}_{\! \vL}$, which is a translation-bounded measure with
well-defined autocorrelation $\gamma$ and diffraction
$\widehat{\gamma}$.  The latter does not depend on the choice of
$\vL$, and is given by
 \begin{equation}\label{eq:F-diff}
      \widehat{\gamma} \, = \sum_{k\in L^{\circledast}}
      I(k) \, \delta^{}_{k}   \quad \text{with} \quad
      I(k) \, = \, \dens (\vL)^2 \sinc \bigl( \pi \tau k^{\star} \bigr)^2 ,
 \end{equation}
 where $L^{\circledast} = \ZZ[\tau]/\text{\small $\sqrt{5}$}\subset
 \QQ(\text{\small $\sqrt{5}$}\, )$ is the \emph{Fourier module},
 $k^{\star}$ is obtained from $k$ by replacing $\sqrt{5}$ with
 $-\sqrt{5}$, and $\sinc (x) = \frac{\sin (x)}{x}$; see
 \cite[Sec.~9.4.1]{TAO} for a derivation. Note that $L^{\circledast}$
 is also the dynamical spectrum (in additive notation) of the
 dynamical system $(\YY,\RR)$, which is pure point with all
 eigenfunctions being continuous \cite{Boris,Daniel}.
 
 Behind this description lies the fact that the Delone set of our
 special fixed point is the \emph{regular model set} of
 Example~\ref{ex:Fibo}, with CPS $(\RR, \RR, \cL)$ and a half-open
 window of length $\tau$, where our choice of the lattice is
 $\cL = \{ (x, x^{\star}) : x \in \ZZ[\tau] \}$, which is the
 Minkowski embedding of $\ZZ[\tau]$ as a planar lattice in
 $\RR \times \RR \simeq \RR^2$. The $\star$-map then is algebraic
 conjugation in the quadratic field $\QQ ( \mbox{\small $\sqrt{5}$} \,
 )$, as induced by $\sqrt{5} \mapsto - \sqrt{5}$.

 More generally, if we use the same CPS with an interval of length $s$
 as window, the diffraction measure of the resulting point set is of
 the same form as in Eq.~\eqref{eq:F-diff}, with the intensity now
 being $I (k) = I (0) \pts \sinc (\pi s \pts k^{\star})^2$.  If one
 considers a sequence of positions $ ( k /\tau^{\ell} )$ with $k \in
 L^{\circledast}$, it is clear that $I ( k / \tau^{\ell} ) = \cO
 \bigl( \tau^{-2\ell} \bigr)$ as $\ell\to\infty$, because $\sinc (x) =
 \cO (x^{-1} )$ as $x \to\infty$. As we shall discuss later in
 Remark~\ref{rem:generic}, this leads to an asymptotic behaviour of
 the form $Z(k) = \cO (k^2)$ as $k\,\raisebox{2pt}{$\scriptscriptstyle
  \searrow$}\, 0$.

 In what follows, we shall see that the decay of $Z(k)$ from
 Eq.~\eqref{eq:Z-def} can be faster, provided $s$ takes special
 values.  This result is implicit in \cite{GL} and was recently
 re-derived, by slightly different methods, in \cite{Josh1}.

\begin{prop}\label{prop:module-scale}
  Let\/ $\vL = \oplam (W)$ be a regular model set in the CPS\/
  $(\RR, \RR, \cL)$ of the Fibonacci chain, where the window\/
  $W\subset \RR$ is an interval of length\/ $s\in \ZZ[\tau]$. Then,
  with\/ $L^{\circledast} = \ZZ[\tau]/\mbox{\small $\sqrt{5}$}$ as
  above, the diffraction measure of\/ $\delta^{}_{\! \vL}$ is given
  by\/
  $\widehat{\gamma} = I(0) \sum_{k\in L^{\circledast} } \sinc
  \bigl(\pi s \pts k^{\star} \bigr)^2$, and the intensities along a
  sequence\/ $( k/\tau^{\ell})^{}_{\ell\in\NN_0}$, for any\/
  $0\ne k\in L^{\circledast}$, decay like $\tau^{-4\ell}$ as\/
  $\ell\to\infty$.
\end{prop}

\begin{proof}
  The statement on the diffraction measure, $\widehat{\gamma}$, is a
  direct consequence of the results in \cite[Sec.~9.4.1]{TAO}. It does
  not depend on the position of the window, or on whether the interval
  is closed, open or half-open.
  
 Now,  $s\in\ZZ[\tau]$ means $s=a+b \tau$ with $a,b\in\ZZ$.
 Let $0\ne k\in L^{\circledast}$, so 
 $k = \kappa/\mbox{\small $\sqrt{5}$}$ with 
 $\kappa = m + n \tau$ for some $m,n \in \ZZ$, 
 excluding $m=n=0$. Applying the $\star$-map then gives
\[
    I \Bigl( \myfrac{k}{\tau^{\ell}}\Bigr) \, = \, \dens (\vL)^2 \,
    \sinc \Bigl( \myfrac{\pi \tau^{\ell} s \pts 
        \kappa^{\star}}{\mbox{\small $\sqrt{5}$}}
    \Bigr)^2 ,
\] 
with $\ell\in\NN_0$. To continue, we recall the relation
\begin{equation}\label{eq:sin-red}
     \sin (m \pi + x) \, = \, (-1)^m \sin (x)
     \qquad \text{for $m\in\ZZ$ and $x\in\RR$}.
\end{equation}
Moreover, with $f_n$ denoting the Fibonacci numbers as
defined by $f^{}_{0} = 0$, $f^{}_{1} = 1$ together with the
recursion $f^{}_{n+1} = f^{}_{n} + f^{}_{n-1}$, we will need
the well-known formula
\begin{equation}\label{eq:fib-form}
      f^{}_n \, = \, \myfrac{1}{\mbox{\small $\sqrt{5}$}} \,
      \Bigl( \tau^n - \bigl( -1/\tau\bigr)^n \Bigr) ,
\end{equation}
which holds for all $n\in\ZZ$. Then, as $\ell\to\infty$, we get
\[
   \sin \Bigl( \myfrac{\pi \tau^{\ell} s \pts \kappa^{\star}}
    {\mbox{\small $\sqrt{5}$}}
    \Bigr)^2   \,  = \:
    \sin \Bigl( \myfrac{\pi \pts \lvert s \pts 
         \kappa^{\star} \rvert}{\mbox{\small
    $ \sqrt{5}$}} \, \tau^{-\ell}  \Bigr)^2  
    \, = \, \myfrac{\pi^2 \bigl( s \pts \kappa^{\star} \bigr)^2}{5}
      \, \tau^{-2\ell}  \, + \cO \bigl( \tau^{-6 \ell}\bigr) ,
\]
where the first step follows from using Eq.~\eqref{eq:fib-form}
to replace $\tau^{\ell}/\mbox{\small $\sqrt{5}$}$
and then reducing the argument via Eq.~\eqref{eq:sin-red}, which is
possible because all Fibonacci numbers are integers. The second step
uses the Taylor approximation $\sin (x) = x + \cO (x^3)$ for small $x$.

Inserting this expression into the formula for the intensity gives
\begin{equation}\label{eq:fiboint}
     I \Bigl( \myfrac{k}{\tau^{\ell}} \Bigr) \, = \:
     c(s,k) \, \tau^{-4\ell} + \cO \bigl( \tau^{-8 \ell} \bigr) 
     \qquad \text{as $\ell\to\infty$},
\end{equation}
where the constant derives from the previous calculation.
\end{proof}

Analysing the steps of the last proof, one finds
\[
    c (s,k) \, = \, \dens (\vL)^2 \pi^2 s^2 (k^{\star})^2,
\]
where $k^{\star}$ lies in the window and is thus bounded.
So, for any fixed $s \in \ZZ [\tau]$, we have $c (s,k) =
\cO (1)$, which allows us to sum over the inflation series
of peaks as follows. Consider
\[
    \varSigma (k) \, \defeq \sum_{\ell=0}^{\infty}  
    I \Bigl( \myfrac{k}{\tau^{\ell}} \Bigr) ,
    % \, = \, \myfrac{ c (a,b; m,n) \pts \tau^4}{\tau^4-1}
\]
where the previous estimates lead to the asymptotic behaviour
\begin{equation}\label{eq:sig-scal}
    \varSigma \Bigl( \myfrac{k}{\tau^{\ell}} \Bigr) 
    \, \sim \: \widetilde{c} (k) \pts \varSigma (k) \, \tau^{-4\ell}
    \qquad \text{as $\ell \to \infty$} \pts ,
\end{equation}
with $\widetilde{c} (k) = \cO (1)$.
Now, observe that the diffraction measure satisfies
\[
    Z(k) \, = \,
    \widehat{\gamma} \bigl( (0,k \pts ] \bigr) \, = \,
    \sum_{\kappa \in L^{\circledast} \cap 
    (\nts \frac{k}{\tau}, k \pts ]} 
    \varSigma (\kappa) \pts ,
\]
which then, via \eqref{eq:sig-scal}, implies the asymptotic behaviour
\[
      Z \Bigl( \myfrac{k}{\tau^{\ell}} \Bigr) 
      \, \asymp \, Z(k) \, \tau^{-4\ell} .
\]
The upper bound is a consequence of the summability of the geometric
series in conjunction with the boundedness of $\widetilde{c}(k)$,
while the lower bound follows from the existence of at least one
series of non-trivial peaks with a scaling according to
\eqref{eq:fiboint} and the behaviour of $\sinc(x)$ near $x=0$.

Put together, we have the following result.    

\begin{theorem}\label{thm:fibo}
  Under the assumptions of
  Proposition~\textnormal{\ref{prop:module-scale}}, one has
\[
  Z (k) \, \asymp \, k^4  \qquad
  \text{as $k\,\raisebox{2pt}{$\scriptscriptstyle
  \searrow$}\, 0$}.
\]
  In particular, $Z(k)=\cO(k^4)$.\qed
\end{theorem}

\begin{remark}\label{rem:infl}
  The cases with $s\in\ZZ[\tau]$ treated in
  Proposition~\ref{prop:module-scale} are connected in the sense that
  the minimal components\footnote{Recall that a dynamical system
    $(\XX,G)$ is \emph{minimal} when the $G$-orbit of every
    $x\in\XX$ is dense in $\XX$.} of the corresponding hulls are all
  mutually locally derivable (MLD) to each other, so they all have a
  local inflation/deflation symmetry in the sense of
  \cite{TAO}. Within this family, there is the Fibonacci chain, which
  can be constructed from a self-similar tiling inflation rule. It is
  thus natural to assume that the particularly rapid decay of $Z (k)$
  in this case can also be traced back to the inflation nature, as in
  Section~\ref{sec:noble} below.  \exend
\end{remark}   

\begin{remark}\label{rem:generic}
  For general $s$, as mentioned previously, one only has
  $I(k/\tau^{\ell})=\cO\bigl(\tau^{-2\ell}\bigr)$. Since the summation
  argument remains unchanged, one then only gets $Z(k)=\cO(k^2)$ as
  $k\,\raisebox{2pt}{$\scriptscriptstyle \searrow$}\, 0$. Unlike the
  situation treated in Theorem~\ref{thm:fibo}, it is not clear when
  one finds $Z(k)\asymp k^2$. Consequently, it is an open question
  whether intermediate decay exponents are possible, for instance as a
  result of Diophantine approximation properties.  \exend
\end{remark}

\subsection{Noble means inflations}\label{sec:noble}
These are generalisations of the Fibonacci case, with substitution
matrices $M_p=\left(\begin{smallmatrix}p & 1 \\ 1 & 0
    \end{smallmatrix}\right)$ for integer $p\in\NN$, where
  $p=1$ is the Fibonacci example. One possible substitution rule
  \cite[Rem.~4.7]{TAO} is $a\mapsto a^pb$, $b\mapsto a$, but the
  position of the letter $b$ in the image of $a$ does not matter, as
  all choices produce equivalent rules, by an application of
  \cite[Prop.~4.6]{TAO}. All noble means substitutions result in
  Sturmian sequences.

  The inflation factor $\lambda_p$ is the Perron--Frobenius eigenvalue
  of $M_p$, so
  \[
  \lambda_{p}\, =\, \myfrac{1}{2}\bigl(p+\sqrt{p^2+4}\,\bigr)
  \, =\, [p;p,p,p,\ldots]\pts ,
  \]
  with $p=1$ giving the Fibonacci case (golden mean) discussed
  above. The corresponding inflation rule works with lengths
  $\lambda_p$ and $1$ for the intervals of type $a$ and $b$,
  respectively.  Then, the arguments used above generalise directly to
  the noble means inflation case, which possess a model set
  description where both the physical and the internal space are
  one-dimensional, resulting in the same scaling exponents for the
  entire family.

  In fact, the noble means inflations can more easily be analysed
  using the renormalisation approach \cite{BFGR,BGM} mentioned in
  Section~\ref{subsec:renorm}. The two eigenvalues of the inflation
  matrix $M_{p}$ are
  $\lambda_{p}=\frac{1}{2}\big(p+\sqrt{p^2+4}\,\bigr)$ and
  $-1/\lambda_{p}$. Consequently, as explained in
  Section~\ref{subsec:iter}, the Lyapunov spectrum of the matrix
  iteration is
  $\sigma^{}_{\mathrm{L}} (M_p) = \{\log(\lambda_p) , -\log(\lambda_p)
  \}$. To continue, we need the Fourier--Bohr (FB) coefficients for
  the $a$- and $b$-positions separately, which we denote by $A_a (k)$
  and $A_b (k)$, respectively. Note that one can obtain the FB
  coefficients of any weighted version of the chain by an appropriate
  superposition, whose absolute square is the corresponding
  diffraction intensity at $k$.

  Let $A(k) = \bigl(A_a (k), A_b (k) \bigr)^T$ be the vector of FB
  coefficients at a given $k$. Then, the renormalisation approach
  gives the relation
\[
  \big\| A(\lambda^{-n}_{p} k)\big\| \, =\,
  \lambda^{-n}_{p} \pts \big\| B(\lambda^{-n}_{p}k)
  B(\lambda^{-n+1}_{p}k)\dots
  B(\lambda^{-1}_{p}k) \pts A(k)\big\| .
\]
Due to the additional prefactor, by taking logarithms, the Lyapunov
exponents are shifted by $-\log(\lambda_{p})$, and thus are $\{0 ,
-2\log(\lambda_p) \}$. The exponent $0$ is not the relevant one,
because it would correspond to intensities that do not decay as we
approach $k=0$, which contradicts the measure property of
$\widehat{\pts\gamma^{}_{p}\pts}$.  Consequently, the scaling of the
amplitudes is governed by the Lyapunov exponent $-2\log(\lambda_{p})$,
and the corresponding diffraction intensities scale as
 \[
    I(\lambda^{-\ell}_{p}k) \, \sim  \,
    c(k)\pts \lambda^{-4\ell}_{p} I(k)\pts ,
 \]
 which is the generalisation of the behaviour we observed for the
 Fibonacci case in Proposition~\ref{prop:module-scale}. On the basis of
 the same argument for the boundedness of the $c(k)$, we can sum
 the contributions and conclude as follows.

 \begin{coro} For the noble means inflations with\/ $p\in\NN$,
   one has the asymptotic behaviour 
 \[
   Z_p (k) \, = \, \widehat{\pts\gamma^{}_{p}\pts}
   \bigl( (0,k] \bigr) \, \asymp \, k^4
   \quad \text{as $k\,\raisebox{2pt}{$\scriptscriptstyle
    \searrow$}\, 0$.}
 \]
 In particular, $Z_p(k)=\cO(k^4)$.\qed 
\end{coro}

 Note that this result relies on the inflation structure, and thus
 only applies to inflation-symmetric cases, in line with
 Remark~\ref{rem:infl}.  Let us continue with further examples of
 inflation systems, where we state the results in a less formal
 manner.

 \subsection{Limit-periodic systems and beyond}

 Consider the two-letter substitution rule
 \[
   a\mapsto ab\, , \quad b\mapsto aa \, ,
 \]
 with substitution matrix
 $\left(\begin{smallmatrix}1 & 2\\1 & 0 \end{smallmatrix}\right)$.
 The latter is diagonalisable and has eigenvalues $2$ and $-1$. This
 substitution rule defines the \emph{period doubling} system, which is a
 paradigm of a Toeplitz-type system with limit-periodic structure; see
 \cite[Sec.~4.5]{TAO} for background and \cite[Sec.~9.4.4]{TAO} for a
 detailed discussion of its pure point diffraction measure.

 As the period doubling system arises from a primitive substitution
 rule, which in this case is identical to the inflation rule, the
 renormalisation approach of Section~\ref{subsec:renorm} applies.
 Here, the exponents (taking into account the additional factor
 $2^{-n}$ in the amplitude renormalisation equation) are
 $\{0,-\log(2)\}$, of which only the negative one is relevant (for the
 same reason as above). Consequently, the diffraction intensities
 scale as
 \[
     I(2^{-\ell}k) \, \sim \, c(k)\pts 4^{-\ell} I(k) \pts ,
 \]
 and the scaling behaviour for the period doubling system is
 $Z(k)\asymp k^2$ as $k \,\raisebox{2pt}{$\scriptscriptstyle
   \searrow$}\, 0$. Again, this follows via geometric series
 summations in conjunction with the observation that the $c(k)$ are
 bounded (for instance, via the existing model set description; see
 \cite[Ex.~7.4]{TAO}).

 The same type of analysis also works for limit-quasiperiodic systems.
 As an example, consider the two-letter substitution rule $a\mapsto
 aab$, $b\mapsto abab$ from \cite{BMS}, in its proper geometric
 realisation.  The eigenvalues of the substitution matrix are
 $\lambda_{\pm}=2\pm\sqrt{2}$. The standard renormalisation analysis
 from above for the FB coefficient now results in
 \[
   \big\| A(k/\lambda_{+})\big\| \, \sim \,
   \myfrac{2-\sqrt{2}}{2+\sqrt{2}}\, \big\| A(k)\big\| 
 \]
 for $k\,\raisebox{2pt}{$\scriptscriptstyle \searrow$}\, 0$, and thus
 in the scaling behaviour
 $I(\lambda^{-n}_{+}\pts k) \sim c(k)\pts
 \lambda_{+}^{-2n\tilde{\alpha}}$ with
 \[
   \tilde{\alpha} \, = \,
   2 - \frac{\log\pts\lvert\det(M)\rvert}{\log(\lambda_{+})}
   \, = \, 2 \,\frac{\log\bigl(1+\sqrt{2}\,\bigr)}
   {\log\bigl(2+\sqrt{2}\,\bigr)} \, \approx\, 1.436\pts ,
 \]
 where the $c(k)$ are once again bounded.  Consequently, one finds
 $Z(k)\asymp k^{2\tilde{\alpha}}$ as
 $k\,\raisebox{2pt}{$\scriptscriptstyle \searrow$}\, 0$. Our parameter
 $\tilde{\alpha}$ is related to the hyperuniformity parameter $\alpha$
 from \cite{Josh1} by $1+\alpha=2\tilde{\alpha}$.

 More generally, this formula applies to primitive binary inflation
 systems, with the appropriate interpretation of the case
 $\det(M)=0$. We discuss an example of this phenomenon in
 Section~\ref{sec:TM} below. This exceptional scaling behaviour is
 known from \cite{GL}; see also \cite{Josh1}.

\subsection{Substitutions with more than two letters} 

The methods described above are not restricted to the binary case. In
general, one can at least obtain a bound for the \mbox{small-$k$}
behaviour of $Z(k)$ from the spectral gap of the substitution matrix.
For an arbitrary substitution, one generates the corresponding
inflation rule with intervals of natural length, as determined by the
left eigenvector of the substitution matrix.

As an example, let us consider the substitution
\[
  a\mapsto abc\, , \quad
  b\mapsto ab\, , \quad
  c\mapsto b\, ,
\]
which occurs in the description of the Kolakoski sequence with run
lengths $3$ and $1$; see \cite{BS04} as well as \cite[Exs.~4.8 and
7.5]{TAO}.  The substitution matrix has spectrum
$\sigma(M)=\{\lambda,\mu,\overline{\mu}\}$, which are the roots of
the polynomial $x^3-2x^2-1$. In particular, $\lambda$ is a
Pisot--Vijayaraghavan number and a unit, and
$\lvert\mu\rvert^2=1/\lambda$. The natural interval lengths are
$\lambda(\lambda-1)$, $\lambda$ and $1$.  The Lyapunov spectrum is
$\sigma^{}_{\mathrm{L}}(M)=\{\log(\lambda),\log\lvert\mu\rvert\}=
\{\log(\lambda),-\frac{1}{2}\log(\lambda)\}$.  This leads to the
relevant amplitude scaling
\[
  \Big\| A\Bigl(\myfrac{k}{\lambda^{n}}\Bigr) \Big\|\, \sim \,
  c(k)\, \lambda^{-3n/2} \, \big\| A(k)\big\|
\]
as $n\to\infty$. With the usual arguments from above, this gives
\[
  Z(k) \, \asymp \, k^3 \quad\text{as $k\,\raisebox{2pt}
    {$\scriptscriptstyle \searrow$}\, 0\pts$.}
\]
The same asymptotic behaviour for small $k$ emerges for the plastic
number substitution, compare \cite[Exs.~3.5 and 7.6]{TAO}, which is
another ternary substitution whose characteristic polynomial has one
real and a pair of complex conjugate roots.

When one considers a substitution with three real eigenvalues of $M$,
say $\sigma(M)=\{\lambda,\lambda^{}_{1},\lambda^{}_{2}\}$, and
$\lvert\det(M)\rvert=1$, it is not immediately clear which of the two
relevant Lyapunov exponents determines the scaling behaviour. When the
characteristic polynomial is irreducible, it will be the larger of the
two in modulus. In any case, one obtains
$Z(k)=\cO(k^{2\tilde{\alpha}})$ as
$k\,\raisebox{2pt}{$\scriptscriptstyle \searrow$}\, 0$, with
\[
  \tilde{\alpha}\, =\, 1-
  \frac{\log\lvert\lambda_{i}\rvert}{\log(\lambda)}
\]
for $i=1$ or $i=2$. The analogous formula applies to larger alphabets
as well, where it remains to be determined which eigenvalue dominates
the asymptotic behaviour.

\subsection{Square-free integers}

As an example of a pure-point diffractive system of a different nature
(and which, in particular, does not possess an inflation symmetry), we
consider the point set of square-free integers in $\ZZ$.  An integer
is called \emph{square-free} if it is not divisible by the square of
any prime, which for instance means that $1$, $2$, $3$, $5$ and $6$
are square-free, while $0$, $4$, $8$, $9$ and $12$ are not. Clearly,
$n$ is square-free if and only if $-n$ is. Let
\[
    V \, = \, \{ n \in \ZZ : n \text{ is square-free} \}
\]
denote the set of square-free integers, which has density $\dens (V) =
\frac{6}{\pi^2}$. In fact, $V$ can be described by the projection
formalism \cite{BMP}, and is actually an example of a weak model set
of maximal density \cite{BHS,KR}, which gives access to the full
spectral information.

Let us now define the (discrete) hull of $V$ as $\XX_V =
\overline{\{t+V : t \in \ZZ \}}$, where the closure is taken in the
local topology. Then, one considers the topological dynamical system
$(\XX_V , \ZZ)$, which has positive topological entropy.  Some people
prefer to consider the flow under the continuous translation action of
$\RR$, which emerges via a standard suspension with a constant roof
function. By slight abuse of notation, we write $(\XX_V, \RR)$ for it,
although the hull is now different.  One usually equips $\XX_V$ with
the patch frequency or \emph{Mirsky} measure,
$\nu^{}_{\text{M}}$. This system has pure point diffraction and
dynamical spectrum,\footnote{Note that the Kolmogorov--Sinai entropy
  with respect to $\nu^{}_{\mathrm{M}}$ vanishes, as it must for a
  system with pure point spectrum. In particular, the Mirsky measure
  is not a measure of maximal entropy; see \cite{Huck} and references
  therein for more.} and the set $V$ itself is generic for
$\nu^{}_{\text{M}}$. Therefore, one can calculate the diffraction
measure from $V$ itself. It is given \cite{BMP} by
\[
   \widehat{\gamma^{}_{V}} \, = \, \dens(V)^2
   \sum_{\kappa \in L^{\circledast}} I(\kappa) \, \delta^{}_{\kappa} \, ,
   \quad \text{with $\dens (V) = \myfrac{1}{\zeta (2)}
      = \myfrac{6}{\pi^2}$} \pts ,
\]
where $\zeta(s)$ denotes Riemann's zeta function.  Here,
$L^{\circledast}$ is the subset of rational numbers with cube-free
denominator, which is a subgroup of $\QQ$ and, at the same time, the
dynamical spectrum of $(\XX_V, \RR,\nu^{}_{\mathrm{M}})$ in additive
notation.  For $\kappa = \frac{m}{q}$ with $m$ and $q$ coprime
integers, one has $I(\kappa) = I(0) \pts f(q)^2$ with $I(0) = \dens
(V)^2$ and
\begin{equation}\label{eq:f}
     f (q) \, = \, \begin{cases}
     \prod_{p | q} \frac{1}{p^2 -1} \pts , 
     & \text{if $q$ is cube-free}, \\
     0 \pts , & \text{otherwise},  \end{cases}
\end{equation}
where $p | q$ means that $p$ runs through all primes that divide $q$.
In particular, $I (\kappa )$ only depends on the denominator of
$\kappa \in \QQ$, so that the diffraction measure is $1$-periodic.

Let $\NN^{(3)}$ denote the positive cube-free integers and consider $Z
(k) \defeq \widehat{\gamma^{}_{V}} \bigl( (0,k] \bigr) / I(0)$ for small
  $k>0$, where we get
\begin{equation}\label{eq:sqfreeZ}
     Z (k) \; = \sum_{\substack{q \in \NN^{(3)} \\ qk \geqslant 1}}
     \card \{ 1 \leqslant m \leqslant qk : \gcd (m,q)=1 \}
     \, f (q)^2
\end{equation}
because $I \bigl( \frac{m}{q} \bigr)$ with $\gcd (m,q) = 1$ does not
depend on $m$. Via some standard, though slightly tricky, methods from
analytic number theory, one can show \cite{BC} that
\begin{equation}\label{eq:bounds}
  \log\bigl(Z (k)\bigr) \, \sim \,  \myfrac{3}{2} \log(k)
  \quad \text{as $k\,\raisebox{2pt}{$\scriptscriptstyle
  \searrow$}\, 0\pts$}.
\end{equation}

\begin{figure}
\centerline{\includegraphics[width=0.8\textwidth]{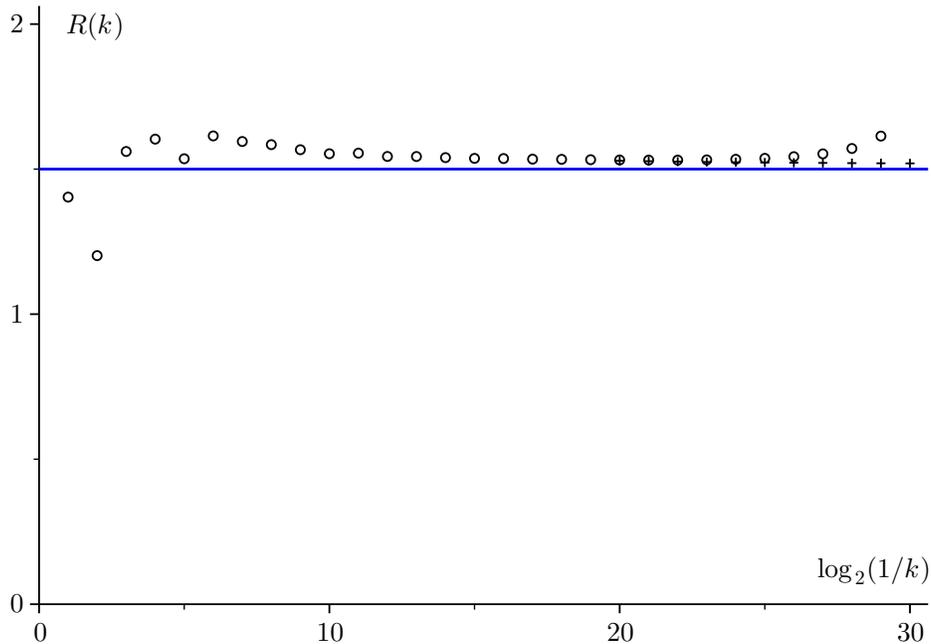}}
\caption{\label{fig:squarefree}Approximation of
  $R(k)=\log(Z(k))/\log(k)$ based on truncating the sum in
  Eq.~\eqref{eq:sqfreeZ} to cube-free numbers $q$ that are multiples of
  the first $2^{13}$ square-free numbers (circles) and the first
  $2^{20}$ square-free numbers (crosses), respectively. The blue line
  corresponds to the ratio $3/2$.}
\end{figure}

The result of a numerical approximation to the sum in
Eq.~\eqref{eq:sqfreeZ} is shown in Figure~\ref{fig:squarefree}. Here,
the sum over all cube-free numbers is rewritten as a sum over all
square-free numbers and their cube-free multiples which share the same
value of the function $f$ of Eq.~\eqref{eq:f}. Truncated sums
including up to $2^{20}$ square-free numbers are used, showing a
behaviour that is close to the asymptotic form of $Z(k)$ from
Eq.~\eqref{eq:bounds}. Note that, due to the truncation of a sum of
positive numbers, the numerical approximations for the ratio
$R(k)=\log(Z(k))/\log(k)$, where
$k\,\raisebox{2pt}{$\scriptscriptstyle \searrow$}\, 0$, are always
larger than the actual values, and apart from a couple of points at
large values of $k$, all numerical estimates are above the asymptotic
line.  Also, since fewer terms contribute to the sums for smaller $k$,
eventually the difference will become large for small $k$, as seen in
Figure~\ref{fig:squarefree} for the data using fewer square-free
numbers. A quantitive estimate of the error is difficult. For a more
detailed analysis of this asymptotic behaviour, including power-free
numbers for larger powers, we refer to \cite{BC}.

\section{Systems with absolutely continuous spectrum}\label{sec:ac}

The generic case for the emergence of absolutely continuous
diffraction is the presence of some degree of randomness. This means
that the most natural setting here is that of stochastic point
processes, and quite a bit of work has recently been done; compare
\cite{GL18} and references therein.

Our goal in this section is modest in the sense that we only aim at
collecting some one-dimensional results that are all well known, but
rarely appear together. Moreover, we are not looking into the
hyperuniform setting as in \cite{GL18}, but rather into systems that
frequently appear in practice. For this reason, we keep also this
section somewhat informal.

\subsection{Poisson and Rudin--Shapiro}

The homogeneous \emph{Poisson process} with intensity (or point
density) $1$ on the real line has diffraction measure
$\widehat{\gamma} = \delta^{}_{0} + \Leb$, where $\Leb$ denotes
Lebesgue measure. This result applies to almost every realisation of
this classic stochastic point process; see \cite[Ex.~11.6]{TAO}.  This
gives
\begin{equation}\label{eq:Poisson}
   Z^{}_{\text{P}} (k) \, = \, k
\end{equation}
in this case, which can serve as a reference for the comparison with
other stochastic processes.

The \emph{Bernoulli lattice gas} on $\ZZ$ with occupation probability
$p$ per site, see \cite[Ex.~11.2]{TAO}, leads to $\widehat{\gamma} =
p^2 \delta^{}_{\ZZ} + p (1-p) \Leb$ for a.e.\ realisation, and hence to 
\[
    Z (k) \, = \, p (1-p) k \pts .
\]
When using weights $\pm 1$ instead of $1$ and $0$, this changes to
$\widehat{\gamma} = (2p-1)^2 \delta^{}_{\ZZ} + 4 p (1-p) \Leb$ and
thus to
\[
      Z (k) \, = \, 4 p (1-p) k
\]
by \cite[Prop.~11.1]{TAO}. In particular, $p=\frac{1}{2}$ gives
$Z (k) = k$ as for the Poisson process above. 
Similar results emerge for more general Bernoulli systems, where the
slope of $Z$ is a measure of the point (or particle) density and of
the variance of the underlying random variable; see
\cite[Rem.~11.1]{TAO} for more.

The binary \emph{Rudin--Shapiro chain}, when realised as a sequence
with weights $\{ \pm 1 \}$, is a deterministic system with diffraction
$\widehat{\gamma} = \Leb$, and thus $Z (k) = k$ exactly; see
\cite[Thm.~10.2]{TAO} and references given there. In particular, this
system has linear complexity, as one can obtain it as a factor of a
substitution rule on four letters \cite[Sec.~4.7.1]{TAO}. Note that
these results apply to every element of the Rudin--Shapiro hull, which
is minimal.

One can apply a simple Bernoullisation process \cite{Bern} to the
Rudin--Shapiro chain (or to any element of its hull). Given $p\in
[0,1]$, one simply changes the sign at each site independently with
probability $p$; see \cite[Sec.~11.2.2]{TAO} for the general setting
and further details. By an application of \cite[Prop.~11.2]{TAO}, the
diffraction measure almost surely does not change, which means that we
get
\[
     Z (k) \, = \, k    \quad \text{for all } p \in [0,1] \pts .
\]
This is an interesting case of \emph{homometry},\footnote{Recall that
  two systems are called \emph{homometric} when they possess the same
  autocorrelation and hence the same diffraction measure; see
  \cite[Sec.~9.6]{TAO} for background.} as one can continuously change
the entropy of the system from $0$ (for $p=0$) to $\log (2)$ (for
$p=\frac{1}{2}$, which really gives the fair coin tossing sequence
with weights $\pm 1$). In particular, the (frequency weighted) patch
complexity has no influence on $Z$ in this case. This highlights the
limitation of this quantity (as well as that of the full diffraction
measure) as a characterisation of order.

\subsection{Random matrix ensembles}

\begin{figure}
  \includegraphics[width=0.4\textwidth]{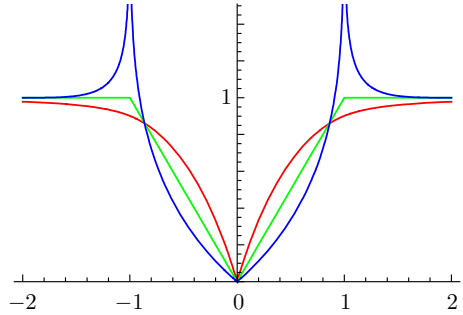}
  \caption{\label{fig:dyson}Absolutely continuous part of the
    diffraction measure for the classic random matrix ensembles with
    parameters $\beta=1$ (red), $\beta=2$ (green) and $\beta=4$
    (blue). }
\end{figure}

The point processes (with point density $1$) derived from the classic
random matrix ensembles with parameter $\beta$, compare \cite{F} for
background, can also be analysed for the scaling near $0$. Based on
the calculations in \cite{BK11}, see also \cite[Thm.~11.3]{TAO}, the
scaling behaviour emerges from integrating the Radon--Nikodym
densities, which are illustrated in Figure~\ref{fig:dyson}. The result
reads
\[
  Z_{\beta}(k) \, = \, \begin{cases}
    k^2 -\tfrac{2}{3}k^3 + \cO(k^4), &
    \text{for $\beta=1$},\\
    \tfrac{1}{2} k^2, & \text{for $\beta=2$},\\
    \tfrac{1}{4}k^2 +\tfrac{1}{12} k^3 + \cO(k^4), &
    \text{for $\beta=4$}.
    \end{cases}
  \]
  The general behaviour is
  $Z_{\beta}(k)=\frac{1}{\beta} k^2+\cO(k^3)$, which significantly
  differs from the other cases discussed in this section.  As
  $\beta\,\raisebox{2pt}{$\scriptscriptstyle \searrow$}\, 0$, the
  leading term diverges, which indicates that one gets a different
  power law in this limit. Indeed,
  $\beta\,\raisebox{2pt}{$\scriptscriptstyle \searrow$}\, 0$ gives the
  homogeneous Poisson process with unit density, with
  $Z^{}_{\text{P}}(k)=k$ as in Eq.~\eqref{eq:Poisson}. In contrast,
  for $\beta\to\infty$, the leading term vanishes. This limit
  corresponds to the (deterministic) point process that realises the
  lattice $\ZZ$, where $Z(k)=0$ for all sufficiently small $k>0$.

\subsection{Markov lattice gas}

Let us consider a simple one-dimensional Markov lattice gas on $\ZZ$,
as defined by
\[
\raisebox{-18pt}{\includegraphics[width=0.3\textwidth]{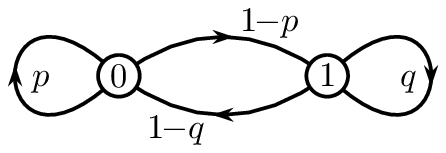}}\qquad
\text{with Markov matrix} \quad 
M \, = \, \begin{pmatrix} p & 1-p \\ 1-q & q \end{pmatrix},
\]
where $p$ and $q$ are probabilities subject to the constraint
$0<p+q<2$.  The two states of the Markov chain correspond to empty
(weight $0$) and occupied (weight $1$) sites. Interpreted as a lattice
gas, we almost surely get a particle system of density
$\varrho=\frac{1-p}{2-p-q}$.

Employing the results from \cite[Ex.~11.11]{TAO}, we know that the
absolutely continuous part of the diffraction measure,
$\widehat{\gamma}^{}_{\text{ac}}$, is almost surely represented by the
Radon--Nikodym density
\[
   g(k) \, = \, \frac{(1-p)(1-q)(1+r)}
                {(1-r)(1-2\pts r\cos(2\pi k) + r^2)}\, ,
\]
where $r=p+q-1$, which satisfies $|r|<1$ by our assumption. A
straightforward calculation gives
\begin{equation}\label{eq:MarkovZ}
  Z(k) \, = \, g(0) \Bigl( k - \myfrac{4\pts\pi^2}{3}\pts
  \myfrac{r}{(1-r)^2} \, k^3 + \cO (k^5)\Bigr)
\end{equation}
with $g(0)=\frac{(1-p)(1-q)(1+r)}{(1-r)^3}\geqslant 0$. When $r=0$,
Eq.~\eqref{eq:MarkovZ} simplifies to
\[
  Z(k)\, =\, (1-p)(1-q)\pts k\pts .
\]
Note that $r<0$ and $r>0$ correspond to effectively repulsive and
attractive systems, respectively; see Figure~\ref{fig:Markov} for an
illustration of $g$ for these two situations.

\begin{figure}
\centerline{\includegraphics[width=0.84\textwidth]{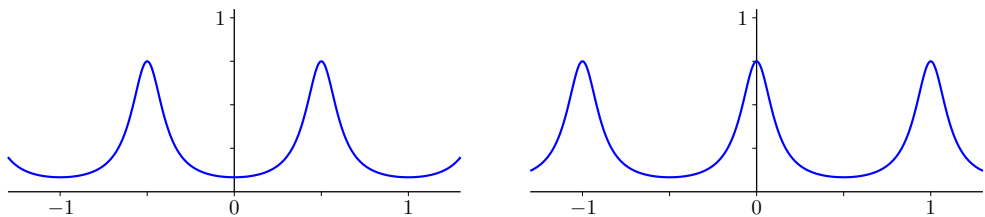}}
\caption{\label{fig:Markov}Absolutely continuous part of the Markov
  chain diffraction. The left panel shows the case $p=q=\frac{1}{4}$,
  which is effectively repulsive, with parameter
  \mbox{$r=-\frac{1}{2}$}. The right panel shows the corresponding
  effectively attractive case for $p=q=\frac{3}{4}$, with
  $r=\frac{1}{2}$.}
\end{figure}

\subsection{Binary random tilings on the line}

To keep things simple, let us consider a random tiling of $\RR$ that
is built from just two intervals (of lengths $u$ and $v$, say) with
corresponding probabilities $p$ and $q=1-p$. We place a scatterer of
unit strength at the left endpoint of each interval. Then, almost
surely, the diffraction comprises a pure point and an absolutely
continuous component, see \cite[Thm.~11.6]{TAO}, where only the latter
contributes to the scaling of $Z(k)$ near $0$. By a straightforward
integration, as $k\,\raisebox{2pt}{$\scriptscriptstyle
  \searrow$}\, 0$, one obtains
\[
  Z(k) \, = \, \frac{p\pts q(u-v)^2}{(pu+qv)^3}
  \left(k+\frac{\pi^2}{9}\pts\frac{u^2 v^2 - 2 u\pts v (pu^2 + qv^2 )}
    {p\pts q(u-v)^2 - (pu^2+qv^2)}\, k^3 + \cO(k^5)\right),
\]
which is non-trivial only for $u\ne v$. For $u=v$, the resulting point
set is always periodic, and $Z(k)=0$ for all sufficiently small $k>0$.

Note that, as in all other examples of this section, $Z(k)$ is a
continuous function, and one can often determine both the scaling
exponent and the corresponding coefficient explicitly. While the
continuity persists for singular continuous $\widehat{\gamma}$,
determining the constant will generally be more complicated.

\section{Singular continuous measures}\label{sec:sc}

Let us finally consider the case of structures with singular
continuous diffraction. While this seems a rare phenomenon in
practice, such systems are actually generic in some sense (in
particular among substitution systems), and thus deserve more
attention.

\subsection{The \mbox{Thue{\pts}}--Morse measure}\label{sec:TM}

The singular continuous \mbox{Thue{\pts}}--Morse (TM) measure
$\mu^{}_{\mathrm{TM}}$ is defined as an infinite Riesz product
\cite{Q,TAO},
\[
  \mu^{}_{\mathrm{TM}} \, = \prod_{\ell=0}^{\infty}
    \bigl( 1 - \cos(2^{\ell+1} \pi x) \bigr) ,
\]
which is to be understood as a limit in the vague topology; we refer
to \cite[Ch.~10.1]{TAO} and references therein for background and
notation.  Note that $\mu^{}_{\mathrm{TM}}$ is a probability measure
on $\TT =\RR/\ZZ$, which is the viewpoint from dynamical systems. It
is connected to the diffraction measure
$\widehat{\gamma}^{}_{\mathrm{TM}}$ on $\RR$ via
\[
     \widehat{\gamma}^{}_{\mathrm{TM}} \, = \,
     \mu^{}_{\mathrm{TM}} * \delta^{}_{\ZZ} \pts ,
\]
under the obvious identification of $\TT$ with $[0,1)$ and addition
modulo $1$. For the investigation of $Z (k)$ for small $k$, we can
thus simply work with $\mu^{}_{\mathrm{TM}}$.

More concretely, for $n\in\NN_0$, we consider the functions
\[
     f^{}_{n} (x) \, = \prod_{\ell=0}^{n-1}
     \bigl( 1 - \cos(2^{\ell+1} \pi x) \bigr) ,
\]
which are probability densities on the unit interval, where
$f^{}_{0} (x) =1$ as usual. They satisfy 
\begin{equation}\label{eq:f-split}
    f^{}_{n+m} (x) \, = \, f^{}_{n} (x) \, f^{}_{m} (2^n x) 
\end{equation}
for all $n,m \in \NN_0$. Since $f^{}_{n} (x) = f^{}_{n} (1-x)$,
one also has
\begin{equation}\label{eq:expect}
     \int_{0}^{1} \! x \, f^{}_{n} (x) \dd x \, = \, \myfrac{1}{2}
\end{equation}
for all $n\in\NN_0$.
The corresponding distribution functions 
\[
    F^{}_{n} (k) \, = \int_{0}^{k} \! \dd F^{}_{n} (x)
    \, =  \int_{0}^{k}  \! f^{}_{n} (x) \dd x 
\]
converge uniformly to the distribution function
$F (x) \defeq \mu^{}_{\mathrm{TM}} \bigl( [0,k] \bigr)$, which is a
continuous function on $[0,1]$ with $F (0)=0$ and $F(1)=1$.  In fact,
$F$ is strictly increasing and satisfies the symmetry relation
$F(1-k) = 1- F(k)$. Moreover, $F$ possesses the uniformly converging
Fourier series representation \cite[Eq.~10.9]{TAO}
\[
    F (k) \, = \, k \, + \sum_{m=1}^{\infty}
    \frac{\eta (m)}{m \pi} \pts \sin (2 \pi m k) \pts ,
\]
where the $\eta(m)$ are the Fourier coefficients of
$\mu^{}_{\mathrm{TM}}$. Note that the $\eta(m)$ are also the pair
correlation coefficients of the \mbox{Thue{\pts}}--Morse sequence when
realised with weights $\pm 1$.  They are recursively specified by
$\eta(0)=1$ together with
\[
  \eta(2m) \, =\,  \eta(m)\quad\text{and}\quad
  \eta(2m+1) \, = \, -\myfrac{1}{2}
  \bigl( \eta (m) + \eta(m+1)\bigr)
\]
for $m\in\NN_0$. In particular,
this entails $\eta(1) = - \frac{1}{3}$, while all other values then
follow recursively.

To study $Z(k) = F(k)$ for small $k$, compare \cite{BGKS}, it suffices
to analyse the behaviour of $F$ along sequences
$(2^{-n})^{}_{n\in\NN}$ because $F$ is strictly increasing and
continuous. Let $N\geqslant n$ and observe
\begin{equation}\label{eq:upper}
\begin{split}
  F^{}_{N} (2^{-n}) \, & = \int_{0}^{2^{-n}} \!\!\!
    f^{}_{N} (x) \dd x
    \, \stackrel{\eqref{eq:f-split}}{=} 
    \int_{0}^{2^{-n}} \!\!\! f^{}_{n} (x) \,
    f^{}_{N-n} (2^n x) \dd x \\[2mm]
    & \leqslant \sup \big\{ f^{}_{n} (x) :
    x\in [0,2^{-n}] \big\}  \;  2^{-n} \!
    \int_{0}^{1} \! f^{}_{N-n} (y) \dd y \,
    = \, 2^{-n}  f^{}_{n} (2^{-n}) \pts ,
\end{split}
\end{equation}
where the last step uses that $f^{}_{n}$ is increasing on
$[0,2^{-n}]$.  In the other direction, one gets
\begin{equation}\label{eq:lower}
\begin{split}
    F^{}_{N} (2^{-n}) \, & = \, 2^{-n} \! \int_{0}^{1}
     f^{}_{n} ( 2^{-n} y) \, f^{}_{N-n} (y) \dd y \\[2mm]
     & \geqslant \, 2^{-n} f^{}_{n} \biggl( 2^{-n} \! 
     \int_{0}^{1} \! y \, f^{}_{N-n} (y) \dd y \biggr) 
     \, \stackrel{\eqref{eq:expect}}{=} \,
     2^{-n} f^{}_{n} (2^{-n-1}) \pts ,
\end{split}
\end{equation}
where the estimate follows from Jensen's inequality, as
$f^{}_{n} (2^{-n} x)$ is convex on $[0,1]$.

\begin{remark}\label{rem:improve}
  One can improve the lower bound by splitting the integral as
\[
    F^{}_{N} (2^{-n}) \,  = \, 2^{-n} \! \int_{0}^{\frac{1}{2}}
     f^{}_{n-1} ( 2^{-n+1} y) \, 2 \pts f^{}_{N-n+1} (y) \dd y
     \, \geqslant \, 2^{-n} f^{}_{n-1} (2^{-n+1} \beta^{}_{N-n+1} )
\]  
   with $\beta^{}_{N-n+1} = 2 \int_{0}^{1/2} y \pts 
   f^{}_{N-n+1} (y) \dd y$. This trick uses the fact that
   $2 f^{}_{n} (x)$ is a probability density on 
   $\bigl[ 0, \frac{1}{2}\bigr]$, as a consequence of
   the reflection symmetry.
   Taking the limit as $N\to\infty$, one
   obtains the estimate $F (2^{-n}) \geqslant 2^{-n} f^{}_{n-1} 
   (2^{-n+1} \beta)$ with
 \[
    \beta \, = \, \myfrac{1}{4} - \myfrac{2}{\pi^2} 
    \sum_{m \geqslant 0} \frac{\eta (2 m +1)}{(2 m + 1)^2}
    \, \approx \, 0.309 94 \pts ,
 \]  
    as can be calculated with the Fourier series representation
    of $F$.
\exend
\end{remark}

So far, by taking $N\to\infty$, we have the bounds
\begin{equation}\label{eq:TM-bounds}
    2^{-n} f^{}_{n} (2^{-n-1}) \, \leqslant \,
    F (2^{-n}) \, \leqslant \, 2^{-n} f^{}_{n} (2^{-n}) \pts ,
\end{equation}
where the lower one can be improved as explained in
Remark~\ref{rem:improve}. For small $x$, a series expansion
analysis of $f^{}_{n}$ gives
\[
    f^{}_{n} (x) \, = \, 2^{n^2} (\pi x)^{2n} \Bigl( 1 -
    \myfrac{4^n - 1}{9} (\pi x)^2 +
    \myfrac{11\cdot 16^n - 25 \cdot 4^n + 14}{2015}
    (\pi x)^4 + \cO \bigl(x^6 \bigr) \Bigr) .
\]
With a bit more work, for $x=2^{-n}$, one can now establish the
asymptotic behaviour
\begin{equation}\label{eq:f-upper}
       2^{-n} f^{}_{n} (2^{-n}) \, \sim \, c \, 2^{-n^2} 
       \Bigl(\myfrac{\pi^2}{2} \Bigr)^{n} 
       \qquad \text{as $n\to\infty$},
\end{equation}
with $c \approx 0.306 {\pts} 663$. The constant was calculated
numerically from the quickly converging sequence of quotients.  Note
that the right-hand side of Eq.~\eqref{eq:f-upper} decays faster than
any power of $2^{-n}$, which means that we will not get a power law
for the scaling of $F(k)$ as $k\,\raisebox{2pt}{$\scriptscriptstyle
  \searrow$}\, 0$.  Similarly, one obtains the asymptotic behaviour
\begin{equation}\label{eq:f-lower}
       2^{-n} f^{}_{n} (2^{-n-1}) \, \sim \,
        \myfrac{\pi^2 c}{4} \, 2^{-n^2} \Bigl( 
        \myfrac{\pi^2}{8} \Bigr)^{n}
       \qquad \text{as $n\to\infty$},
\end{equation}
with the same constant $c$, hence 
$\frac{\pi^2 c}{4} \approx 0.756 {\pts} 660$.

It remains to express the result in terms of $k = 2^{-n}$, via
$\log (k) = -n \log (2)$. From Eqs.~\eqref{eq:f-upper} and
\eqref{eq:f-lower}, in conjunction with the continuity and
monotonicity of $F$, one obtains the following result.

\begin{theorem}\label{thm:tm}
  The distribution function\/ $F$ of the \mbox{Thue{\pts}}--Morse
  measure satisfies the asymptotic bounds
\[
     c^{}_{1} k^{2+\alpha} \exp \Bigl(
     - \myfrac{\log (k)^2}{\log (2)} \Bigr)
     \, \leqslant \, F (k) \, \leqslant \,
     c^{}_{2} \pts k^{\alpha} \pts \exp \Bigl(
     - \myfrac{\log (k)^2}{\log (2)} \Bigr)
\]
for small\/ $k>0$, with exponent
\[
    \alpha \, = \, - \frac{\log \bigl( 
    \frac{\pi^2}{2} \bigr)}{\log(2)} \, \approx \,
    -2.302 {\pts} 992 
\]
and suitable constants $c^{}_{1}$, $c^{}_{2}$.\qed
\end{theorem}
    
A similar estimate, with the same leading term, was previously given
in \cite{GL}. Some numerical experiments indicate that the upper bound
is better then the lower one (which can be improved by
Remark~\ref{rem:improve}), and that the true behaviour is closer to a
prefactor of the form $k^{\alpha+s}$ with $s$ near $\frac{1}{2}$,
which deserves further analysis.

\subsection{Generalised \mbox{Thue{\pts}}--Morse (gTM) sequences}

Let us finally consider the family of binary, primitive,
constant-length substitutions
$\varrho^{}_{p,q} \! : a \mapsto a^p b^q, b \mapsto b^{\pts p} a^q$
with $p,q\in\NN$; see \cite{BGG} for a general exposition and
\cite[Ex.~10.1]{TAO} for a short summary.  The case $p=q=1$ is the
classic \mbox{Thue{\pts}}--Morse substitution considered above. In
what follows, we suppress the explicit dependence on $p$ and $q$ for
ease of notation.

\begin{figure}
\includegraphics[width=0.9\textwidth]{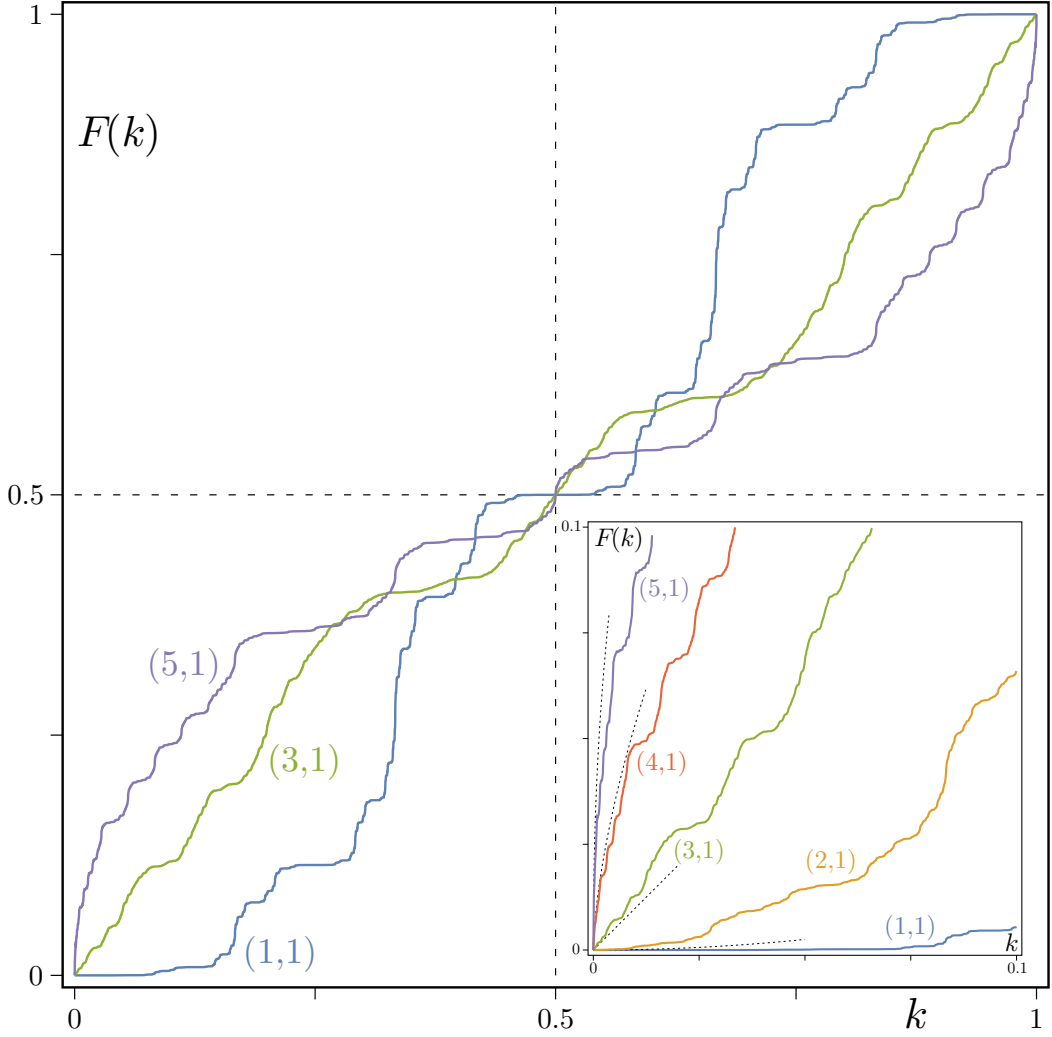}
\caption{Distribution functions $F(k)$ of the gTM measures for
  $(p,1)$ with $p\in\{1,3,5\}$ and the behaviour at small $k$ for
  $p\in\{1,2,3,4,5\}$ (inset).\label{fig:TM} Dashed lines indicate the
  scaling behaviour from Eq.~\eqref{eq:gtmscaling} for
  $p\in\{2,3,4,5\}$.}
\end{figure}

Here, the spectral measure is given by the Riesz product
$\mu  = \prod_{m\geqslant 0} \vartheta\bigl((p+q)^m x\bigr)$ with the 
non-negative trigonometric polynomial
\begin{equation}\label{eq:thetadef}
  \vartheta(x) \, = \, 1 + \myfrac{2}{p+q} \sum_{r=1}^{p+q-1}
  \alpha (p,q,r) \, \cos(2\pi rx) \pts ,
\end{equation}
where $\alpha (p,q,r)=p+q-r-2\pts\min(p,q,r,p+q-r)$. For $x\to 0$,
$\vartheta$ behaves as 
\begin{equation}\label{eq:thetaexp}
  \vartheta(x) \, = \, \frac{(p-q)^2}{p+q} + 
  \frac{12\pts p^2 q^2 + (p-q)^2 - (p-q)^4}{3(p+q)} 
   \pts(\pi x)^2 + \cO (x^4)\pts .
\end{equation}
Whenever $p=q$, we have $\vartheta(0)=0$, and we are in a situation
similar to that of the classic \mbox{Thue{\pts}}--Morse
sequence. Otherwise, $\vartheta(0)>0$, where this value can be less
than $1$, precisely $1$ or larger than $1$. Nevertheless, all these
cases can be treated by the same approach. Defining
$f^{}_{n}(x)=\prod_{m=0}^{n-1}\vartheta\bigl((p+q)^mx\bigr)$, one has
\[
   f^{}_{n+m}(x) \, = \, f^{}_{n}(x)\, f^{}_{m}\bigl((p+q)^n x\bigr)
\]
for arbitrary $m,n\in\NN_{0}$. In particular, this gives
\[
   f^{}_{n+m}\Bigl(\myfrac{x}{(p+q)^n}\Bigr) \, = \, 
   \vartheta\Bigl(\myfrac{x}{(p+q)^n}\Bigr) \, 
   f^{}_{n+m-1}\Bigl(\myfrac{x}{(p+q)^{n-1}}\Bigr).
\]
 
Since $\lim_{n\to\infty}\vartheta\bigl((p+q)^{-n}x\bigr)=\vartheta(0)$,
we see that, for any fixed $m$, the limit
\[
   \lim_{n\to\infty} \vartheta(0)^{-n} f^{}_{n+m}\bigl((p+q)^{-n}x\bigr) 
\]
exists. By standard arguments, this implies the asymptotic behaviour
\[
   F\bigl((p+q)^{-n}\bigr) \, \sim \,  
   \biggl(\frac{p+q}{\vartheta(0)}\biggr)^{\! -n}\qquad
   \text{as $n\to\infty$.}
\]
Since $F$ is continuous and increasing, using
$\vartheta(0)=(p-q)^2/(p+q)$, one finds the scaling behaviour
\begin{equation}\label{eq:gtmscaling}
   Z(k) \, = \, F(k) \, \sim \, k^{2-2\frac{\log|p-q|}{\log(p+q)}}\qquad
   \text{as $k\,\raisebox{2pt}{$\scriptscriptstyle
       \searrow$}\, 0 \pts $,}
\end{equation}
which is in agreement with the scaling argument from \cite{GL} and
Conjecture 3 in \cite{Josh2}. For $p\ne q$, the exponent can be $2$,
$1$ or take positive values below $1$ that can become arbitrarily
small. The corresponding behaviour for small $k$ matches well with the
examples from \cite{BGG} and \cite[Fig.~10.2]{TAO}, as illustrated in
Figure~\ref{fig:TM} for some characteristic choices of the parameters.

\section{Concluding remarks}

The situation for one-dimensional systems is in reasonably good shape,
though a better understanding is still desirable. This is particularly
so for multi-type systems, where the existing approach gives bounds
but not necessarily the exact asymptotic behaviour.

Less clear is the appropriate approach to analyse tilings and point
processes in higher dimension, where one could single out special
directions, but also consider the total diffraction into regions near
the origin. This should be possible within the realm of both inflation
tilings and projection point sets.

\section*{Acknowledgements}

It is our pleasure to thank Michael Coons, Franz G\"{a}hler, Philipp
Gohlke, Neil Ma\~{n}ibo and Joshua Socolar for valuable discussions
and suggestions. We thank the MATRIX Institute in Creswick and the
Department of Mathematics and Physics at the University of Tasmania in
Hobart for hospitality, where this manuscript was completed. This work
was supported by the German Research Foundation (DFG) through CRC 1283
and by EPSRC through grant EP/S010335/1.

This paper is dedicated to the memory of Vladimir Rittenberg, who was
our PhD supervisor. We owe a lot to Vladimir, who always freely shared
his views and insights with his students, but also strongly supported
them in pursuing their own ideas.


\begin{thebibliography}{99}

\bibitem{VR88}
Alcaraz F C, Baake M, Grimm U and Rittenberg V,
Operator content of the XXZ chain,
\textit{J.\ Phys.\ A:\ Math.\ Gen.} \textbf{21} (1988) L117--L120.

\bibitem{VR10}
Alcaraz F C and Rittenberg V,
Shared information in stationary states at criticality,
\textit{J.\ Stat.\ Mech.} (2010) P03024:1--28; 
\texttt{arXiv:0912.2963}.

\bibitem{Aubry}
Aubry S, Godr\`{e}che C and Luck J M,
Scaling properties of a structure intermediate between
quasiperiodic and random, \textit{J.\ Stat.\ Phys.}
\textbf{51} (1988) 1033--1074.

\bibitem{BC}
Baake M and Coons M,
Scaling of the diffraction measure of $k$-free integers near
the origin, \textit{preprint} (2019);
\texttt{arXiv:1904.00279}.

\bibitem{BFGR}
Baake M, Frank N P, Grimm U and Robinson E A,
Geometric properties of a binary non-Pisot inflation
and absence of absolutely continuous diffraction,
\textit{Studia Math.} \textbf{247} (2019) 109--154;
\texttt{arXiv:1706.03976}.

\bibitem{BG15}
Baake M and G\"{a}hler F,
Pair correlations of aperiodic inflation rules via 
renormalisation:\ Some interesting examples,
\textit{Topol.\  \& Appl.} \textbf{205} (2016) 4--27; 
\texttt{arXiv:1511.00885}.

\bibitem{BGG}
Baake M, G\"{a}hler F and Grimm U,
Spectral and topological properties of a family of generalised 
\mbox{Thue{\pts}}--Morse sequences,
\textit{J.\ Math.\ Phys.} \textbf{53} (2012) 032701:1--24;
\texttt{arXiv:1201.1423}.

\bibitem{BGM}
Baake M, G\"{a}hler F and Ma\~{n}ibo N,
Renormalisation of pair correlation measures for primitive 
inflation rules and absence of absolutely continuous diffraction,
\textit{Commun.\ Math.\ Phys.}, in press;
\texttt{arXiv:1805.09650}.

\bibitem{BGKS}
Baake M, Gohlke P, Kesseb\"{o}hmer M and Schindler T,
Scaling properties of the \mbox{Thue{\pts}}--Morse measure,
\textit{Discr.\ Cont.\ Dynam.\ Syst.\ A} \textbf{39} (2019)
4157--4185; \texttt{arXiv:1810:06949}.

\bibitem{Bern}
Baake M and Grimm U,
Kinematic diffraction is insufficient to distinguish
order from disorder, \textit{Phys.\ Rev.\ B} \textbf{79}
(2009) 020203(R):1--4 and \textit{Phys.\ Rev.\ B}
\textbf{80} (2009) 029903(E);
\texttt{arXiv:0810.5750}.

\bibitem{TAO}
Baake M and Grimm U,
\textit{Aperiodic Order. Vol.\ $1$: A Mathematical Invitation},
Cambridge University Press, Cambridge (2013).

\bibitem{TAO2}
Baake M and Grimm U (eds.),
\textit{Aperiodic Order. Vol.\ $2$: Crystallography and
Almost Periodicity}, Cambridge University Press, 
Cambridge (2017).

\bibitem{BaGriMa}
Baake M, Grimm U and Ma\~{n}ibo N,
Spectral analysis of a family of binary inflation rules,
\textit{Lett.\ Math.\ Phys.} \textbf{108} (2018) 1783--1805; 
\texttt{arXiv:1709.09083}.

\bibitem{BHS}
Baake M, Huck C and Strungaru N,
On weak model sets of extremal density,
\textit{Indag.\ Math.} \textbf{28} (2017) 3--31;
\texttt{arXiv:1512.07129}.

\bibitem{BK11}
Baake M and K\"{o}sters H,
Random point sets and their diffraction,
\textit{Philos.\ Mag.} \textbf{91} (2011) 2671--2679;
\texttt{arXiv:1007.3084}.

\bibitem{BL}
Baake M and Lenz D,
Spectral notions of aperiodic order, 
\textit{Discr.\ Cont.\ Dynam.\ Syst.\ S} \textbf{10} (2017) 161--190;
\texttt{arXiv:1601.06629}.

\bibitem{BMP}
Baake M, Moody R V and Pleasants P A B,
Diffraction of visible lattice points and $k$th
power free integers,
\textit{Discr.\ Math.} \textbf{221} (2000) 3--42;
\texttt{arXiv:math.MG/9906132}.

\bibitem{BMS}
Baake M, Moody R V and Schlottmann M,
Limit-(quasi)periodic point sets as quasicrystals with $p$-adic
internal spaces,
\textit{J.\ Phys.\ A:\ Math.\ Gen.} \textbf{31} (1998) 5755--5765;
\texttt{arXiv:math-ph/9901008}.

\bibitem{BS04}
Baake M and Sing B,
Kolakoski-$(3,1)$ is a (deformed) model set,
\textit{Can.\ Math.\ Bull.} \textbf{47} (2004) 168--190;
\texttt{arXiv:math.MG/0206098}.

\bibitem{BP}
Barreira L and Pesin Y,
\textit{Nonuniform Hyperbolicity}, 
Cambridge University Press, Cambridge (2007).

\bibitem{BuSol}
Bufetov A and Solomyak B,
On the modulus of continuity for spectral measures in 
substitution dynamics,
\textit{Adv.\ Math.} \textbf{260} (2014) 84--129;
\texttt{arXiv:1305.7373}.

\bibitem{BuSol2}
Bufetov A and Solomyak B,
A spectral cocycle for substitution systems and translation flows,
\textit{preprint} \texttt{arXiv:1802.04783}.

\bibitem{VR05}
de Gier J, Nichols A, Pyatov P and Rittenberg V,
Magic in the spectra of the XXZ quantum chain with boundaries at 
$\Delta=0$ and $\Delta =-1/2$, 
\textit{Nucl.\ Phys.\ B} \textbf{729} (2005) 387--418;
\texttt{arXiv:hep-th/0505062}.

\bibitem{VR03}
de Gier J, Nienhuis B, Pearce P A and Rittenberg V,
Stochastic processes and conformal invariance,
\textit{Phys.\ Rev.\ E} \textbf{67} (2003) 016101:1--4;
\texttt{arXiv:cond-mat/0205467}.

\bibitem{F}
Forrester P,
\textit{Log-Gases and Random Matrices},
Princeton University Press, Princeton (2010).

\bibitem{GL18} 
Ghosh S and Lebowitz J L, 
Generalized stealthy hyperuniform processes:\ 
Maximal rigidity and the bounded holes conjecture,
\textit{Commun.\ Math.\ Phys.} \textbf{363} (2018) 97--110;
\texttt{arXiv:1707.04328}.

\bibitem{GL}
Godr\`{e}che C and Luck J M,
Multifractal analysis in reciprocal space and the nature of
the Fourier transform of self-similar structures,
\textit{J.\ Phys.\ A:\ Math.\ Gen.} \textbf{23} (1990) 
3769--3797.

\bibitem{Hof}
Hof A,
On diffraction by aperiodic structures,
\textit{Commun.\ Math.\ Phys.} \textbf{169} (1995) 25--43.

\bibitem{KR}
Keller G and Richard C,
Periods and factors of weak model sets,
\textit{Israel J.\ Math.} \textbf{229} (2019) 85--132;
\texttt{arXiv:1702.02383}.

\bibitem{Daniel}
Lenz D,
Continuity of eigenfunctions of uniquely ergodic dynamical
systems and intensity of Bragg peaks,
\textit{Commun.\ Math.\ Phys.} \textbf{287} (2009) 225--258;
\texttt{arXiv:math-ph/0608026}.

\bibitem{Luck}
Luck J M,
A classification of critical phenomena on quasi-crystals and
other aperiodic structures, 
\textit{Europhys.\ Lett.} \textbf{24} (1993) 359--364. 

\bibitem{NM}
Ma\~{n}ibo N,
Spectral analysis of primitive inflation rules,
\textit{Oberwolfach Rep.} \textbf{14} (2017) 2830--2832.

\bibitem{MS}
Moody R V and Strungaru N,
Almost periodic measures and their Fourier transforms,
in \cite{TAO2}, pp.~173--270.

\bibitem{Josh1}
O\u{g}uz E C, Socolar J E S, Steinhardt P J and Torquato S,
Hyperuniformity of quasicrystals,
\textit{Phys.\ Rev.\ B} \textbf{95} (2017) 054119:1--10;
\texttt{arXiv:1612:01975}.

\bibitem{Josh2}
O\u{g}uz E C, Socolar J E S, Steinhardt P J and Torquato S,
Hyperuniformity and anti-hyperuniformity in one-dimensional
substitution tilings,
\textit{Acta Cryst.\ A} \textbf{75} (2019) 3--13;
\texttt{arXiv:1806.10641}. 

\bibitem{Huck}
Pleasants P A B and Huck C,
Entropy and diffraction of the $k$-free points in 
$n$-dimensional lattices,
\textit{Discr.\ Comput.\ Geom.} \textbf{50} (2013) 39--68;
\texttt{arXiv:1112.1629}.

\bibitem{Q}
Queff\'{e}lec M,
\textit{Substitution Dynamical Systems --- Spectral Analysis},
2nd ed., LNM 1294, Springer, Berlin (2010).

\bibitem{Boris}
Solomyak B,
Dynamics of self-similar tilings,
\textit{Ergod.\ Th.\ \& Dynam.\ Syst.} \textbf{17} (1997) 695--738
and \textit{Ergod.\ Th.\ \& Dynam.\ Syst.} \textbf{19} (1999) 
1685 (erratum).

\bibitem{TS}
Torquato S and Stillinger F H, 
Local density fluctuations, hyperuniformity, and order metrics, 
\textit{Phys.\ Rev.\ E} \textbf{68} (2003) 041113:1--25
and \textit{Phys.\ Rev.\ E} \textbf{68} (2003) 069901
(erratum); \texttt{arXiv:cond-mat/0311532}.  

\bibitem{Viana}
Viana M,
\textit{Lectures on Lyapunov Exponents},
Cambridge University Press, Cambridge (2013).

\end{thebibliography}
\end{document}